\documentclass[11pt]{amsart} 

\usepackage{setspace} \onehalfspace
\usepackage[active]{srcltx}

\usepackage[all]{xy} 
\usepackage{amsmath} 
\usepackage{amssymb}
\usepackage{amsthm} \usepackage[pdfauthor={Matei Toma},
pdftitle={Bounded families of sheaves on K\"ahler manifolds},
pdfkeywords={Douady space, bounded family of sheaves, K\"ahler manifold}, pdfview={Fit}]{hyperref}

\newenvironment{pf}{\begin{proof}}{\end{proof}}
\newtheorem{Th}{Theorem}[section] 
\newtheorem*{Thm}{Theorem}
\newtheorem{Lem}[Th]{Lemma}
\newtheorem{Cor}[Th]{Corollary} 
\newtheorem{Prop}[Th]{Proposition}

\newtheorem{Def}[Th]{Definition}

\newtheorem{Rem}[Th]{Remark}

\def\cal{\mathcal}

\newcommand{\cA}{\mathcal{A}} 
\def\cC{\mathcal{C}} 
\def\cD{\mathcal D }
\def\cE{\mathcal E} 
\def\cF{\mathcal F}
\def\cG{\mathcal G}
\newcommand{\cH}{\mathcal{H}}
\newcommand{\cI}{\mathcal{I}}
\newcommand{\cK}{\mathcal{K}}
\newcommand{\cL}{\mathcal{L}}
\newcommand{\cM}{\mathcal{M}}
\newcommand{\cN}{\mathcal{N}}
\newcommand{\cO}{\mathcal{O}}
\newcommand{\cP}{\mathcal{P}}

\def\cT{\mathcal T }

\def\cM{\mathcal M } 
\def\cI{\mathcal I }

\newcommand{\N}{\mathbb{N}}
\renewcommand{\P}{\mathbb{P}} 
\newcommand{\Q}{\mathbb{Q}}
\newcommand{\R}{\mathbb{R}}
\newcommand{\Z}{\mathbb{Z}}

\def\ch{\mbox{ch}}

\def\id{\mbox{id}} 
\def\Im{\mbox{Im}}
\def\Ext{\mbox{Ext}}
\def\Hom{\mbox{Hom}} 
\def\Ker{\mbox{Ker}} 
 
\def\Pic{\mbox{Pic}} 
\def\Quot{\mbox{Quot}}
\def\rank{\mbox{rank}}
 
\def\Sing{\mbox{Sing}}
\def\Supp{\mbox{Supp}}

\def\Tors{\mbox{Tors}}

\def\Td{\mbox{Td}}
\def\vol{\mbox{vol}}

\title{Bounded sets of sheaves on K\"ahler manifolds}
\author{Matei Toma} \date{\today}
\thanks{ AMS
  Classification (2000): 32J99; secondary: 14C05.}

\begin{document}
\maketitle

\begin{abstract}
We show that any set of quotients with fixed Chern classes  of a given coherent sheaf on a compact K\"ahler manifold is bounded in a sense which we define. The result is proved by adapting Grothendieck's boundedness criterion expressed via the Hilbert polynomial to the K\"ahler set-up. As a consequence we obtain the compactness of the connected components of the Douady space of a compact K\"ahler manifold. 
\end{abstract}

\noindent

\section{Introduction}

Let $(X, \cO_X(1))$ be a projective scheme endowed with a very ample line bundle. 
In \cite{Gro61} Grothendieck constructed the Hilbert scheme of $X$ which 
 parametrizes the closed subschemes of $X$. He also showed that
by fixing the Hilbert polynomial of the closed subschemes which are to be parametrized, one gets a projective subscheme of the Hilbert scheme. In particular the connected components of the Hilbert scheme are projective. One of the ingredients of Grothendieck's proof was the introduction of the notion of boundedness for sets of coherent sheaves together with a boundedness criterion which says that such a set is bounded if and only if all its elements are quotients of a unique coherent sheaf and their Hilbert polynomials range through a finite set, \cite{Gro61} Thm. 2.1. This new notion proved to be useful in other contexts as well, such as the construction of moduli spaces of semistable sheaves, cf. \cite{HL}.

In this paper we introduce boundedness for sets of coherent sheaves over complex analytic spaces and prove the following criterion in the K\"ahler case; see sections \ref{sectionchern}, \ref{sectionboundedness} for the definitions.
\begin{Thm} 
Let $(X,\omega)$ be a K\"ahler complex manifold and $F$  a set of isomorphism classes of coherent sheaves on  $X$. 
 Then the set $F$ is bounded if the following conditions are fulfilled:
\begin{enumerate}
\item There exists a bounded set $G$ of isomorphism classes of coherent sheaves on  $X$
 such that each element of $F$ is a quotient of an element of $G$.
 \item The supports of the sheaves of $F$ are contained in some compact subset of $X$. 
\item The degrees of the sheaves of $F$ are bounded from above. 
\end{enumerate}
\end{Thm}

For the proof we basically follow Grothendieck's approach but several changes and new arguments are needed; one important technical tool in the projective case for instance, for which we could find no analogue in the K\"ahler case, is the use of linear projections. The main idea is to reduce the problem to bounding the volume of some appropriate analytic cycles and use the compactness of the associated cycle space provided by Bishop's theorem, cf. \cite{Fuj78}, \cite{Lie77}. 
 
Our main application is a compactness result for the connected components of the Douady space of a compact K\"ahler manifold, section \ref{sectioncor}. To put this into perspective recall that the set of subspaces of a compact analytic space $X$
was endowed with a natural analytic structure by Douady in \cite{Dou66}. Fujiki proved that the topology of this space had a countable basis, \cite{Fuj79}. He further showed that its irreducible components were compact when $X$ was a K\"ahler space and more generally when $X$ belonged to the class $\cC$, \cite{Fuj84}. For the connected components of the space of divisors of a normal compact space $X$ compactness and even projectivity were established by Fujiki in \cite{Fuj82}.

Other applications are properness for the morphism ``Douady $\to$ Barlet'' and  compactness for the moduli space of torsion free rank one sheaves on compact K\"ahler manifolds. We believe that our criterion may be further used to prove compactness of  yet to be constructed moduli spaces of higher rank semistable sheaves over compact  K\"ahler manifolds.

We start by reviewing some properties of the homology Todd class which will serve as a replacement of the Hilbert polynomial. Using semi-analytic Stein compacta  in the next section boundedness for sets of isomorphism classes of coherent sheaves is introduced in the analytic set-up and some basic properties are proven. A boundedness result for reflexive sheaves of rank $1$ on compact K\"ahler normal spaces is also included. We prove the boundedness criterion in section  \ref{main} and we end with applications.

{\em Acknowledgements:} I wish to thank Daniel Barlet and Julien Grivaux for several discussions and particularly Jon Magnusson for arousing my interest in these topics.  Thanks are also due to the referees whose suggestions helped in improving the presentation.

\section{Homology classes and degrees}\label{sectionchern}

The Grothendieck-Riemann-Roch theorem for singular varieties was proved by Baum, Fulton and MacPherson \cite{BFM75}, \cite{BFM79}
in the projective case and by Levy \cite{Lev87} in the complex analytic case. 
One way to formulate it is that there exists a natural transformation of functors $\tau:K_0\to H_{2*}( \ ;\Q)$ such that for any compact complex space $X$ the diagram 
$$
  \xymatrix{K^0X\otimes K_0X \ar[r]^{\otimes}\ar[d]^{\ch\otimes\tau} 
& K_0X \ar[d] ^{\tau}\\
H^{2*}(X;\Q)\otimes H_{2*}(X;\Q) \ar[r]^{ \ \ \ \ \ \ \frown} & H_{2*}(X;\Q)
 }
  $$ 
commutes and if $X$ is nonsingular then $\tau(\cO_X)=\Td(X)\frown[X]$, where
$K^0X$, $K_0X$ are the Grothendieck groups generated by holomorphic vector bundles and coherent sheaves respectively and $\Td(X)$ is the (cohomology) Todd class of the tangent bundle to $X$. Naturality means that for each proper morphism
$f:X\to Y$ of complex spaces the diagram
$$
  \xymatrix{ K_0X \ar[r]^{\tau}\ar[d]^{f_!} 
& H_{2*}(X;\Q) \ar[d] ^{f_*}\\
K_0Y\ar[r]^{\tau} & H_{2*}(Y;\Q)
 }
  $$ 
commutes, where $f_!$ is defined by $f_!([\cF])=\sum_i(-1)^i [R^if_*(\cF)]$ for any coherent sheaf $\cF$ on $X$. (In the non-compact case $\tau$ takes values in the Borel-Moore homology, \cite{I}, \cite{Ful} 19.1.) We refer to the original papers and to the books \cite{Ful}, \cite{FL}, \cite{DV} for a thorough treatment of these facts.

For a coherent sheaf $\cF$ on a compact complex space $X$ we shall call $\tau(\cF):=\tau([\cF])$ the {\em homology Todd class} of $\cF$. 
We list some of its properties :
\begin{enumerate}
\item If $\cF$ is locally free and $X$ is smooth and connected, $\tau(\cF)$ is the Poincar\'e dual of $ch(\cF)\cdot\Td(X)\in H^{2*}(X;\Q)$.
\item If $f:X\to Y$ is an embedding then $\tau(f_*\cF)=f_*(\tau(\cF))$.
\item $\tau(\cF)_r=0$ for $r>\dim \Supp\cF$.
 \item $\tau$ is additive on exact sequences.
\item If $X$ is irreducible then $\tau(\cF)_{\dim X}=\rank(\cF)[X]\in H_{2\dim X}(X;\Q) $.

(This may be deduced from the previous properties via some modification of $X$ which desingularizes $X$ and flattens $\cF$.) 
\item  The component of  $\tau(\cF)$ in degree $\dim \Supp\cF$ is the homology class of an effective analytic cycle.

(In order to see this one makes the following reduction steps. For any integer  $r$ we set $\cN_r(\cF)$ to be the sheaf of sections of $\cF$ whose support have dimension less than $r$ and $\cF_{(r)}:=\cF/\cN_r(\cF)$. If $r=\dim \Supp\cF$ we may assume that $\cF=\cF_{(r)}$. Let now $\cI$ be the reduced ideal of $\Supp\cF$. We may also work with the graduation of 
$0=\cI^k\cF \subset \cI^{k-1}\cF\subset ...\subset \cI\cF\subset\cF$ instead of $\cF$ and in particular we may suppose that $\cI\cF=0$. If $Y_1, ..., Y_n$ are the irreducible  components of $\Supp\cF$ and $\iota_j:Y_j\to X$ their respective embeddings, we consider $\cF_j:=\Im(\cF\to\iota_{j*}\iota_{j}^*\cF)$. The assumption $\cF=\cF_{(r)}$ now implies that $\cF$
embeds into $\oplus_j\cF_j$ and $\dim\Supp((\oplus_j\cF_j)/\cF)<r$ thus reducing the situation to the case of an irreducible support.)
\end{enumerate}

In the complex analytic setting a definition of the homology Todd class may be given using  real-analytic locally free resolutions; see  also \cite{AtHi61} for the analogous construction of the cohomological {\em Grothendieck element}
$\gamma_X(\cF)$.  Using this approach one sees easily that property (1) above holds also for 
arbitrary coherent sheaves when $X$ is smooth and connected. This allows one to prove the following result on the variation of the homology Todd classes in a flat family.

\begin{Prop} \label{plat}
 Let $X$ and $S$ be complex manifolds with $X$ compact and $S$ connected and $\cF$ be a coherent sheaf on $X\times S $ flat over $S$.
Then  the class $\tau(\cF_s)\in  H_{2*}(X;\Q)$ is independent of $s\in S$.
\end{Prop}

Let $(X,\omega)$ be a K\"ahler compact analytic space. The class $[\omega]\in H^2(X,\R)$ allows us to define 
{\em degrees} 
of a coherent sheaf $\cF$ on $X$ in the following way. For each $r\in\N$ we define the $r$-degree of $\cF$ by
$$
\deg_r(\cF):=[\omega^r]\frown\tau_{r}(\cF)\in \R.
$$
Notice that in case $(X, \cO_X(1))$ is polarized, smooth, projective and $[\omega]=c_1(\cO_X(1))$ one recovers the coefficient of the Hilbert polynomial of $\cF$ in degree $r$ as 
$$
\frac{\deg_r(\cF)}{r!}.
$$

\section{Bounded sets of coherent sheaves}\label{sectionboundedness}

According to Grothendieck's definition a family of coherent sheaves on a scheme $X$ (over a field $k$ for instance) is {\em bounded} if it can be parametrized by a scheme of finite type over $k$. One feature of this context is that the topology of the parameter space is  noetherian, i.e. every descending chain of closed subspaces is stationary. In the complex analytic setting a possible substitute would be to consider compact parameter spaces and work with their Zariski topology. We chose to work with semi-analytic Stein compacta, which are better adapted to our purposes; see 
\cite{BS} 5.1.f for the definition. Their Zariski topology is known to be noetherian, cf. 
\cite{BS} 5.3, p.220.

\begin{Def}\label{def:bounded}
Let $X$ be an analytic space 
over an analytic space $S$ and $E$ a set of isomorphism classes of coherent sheaves on the fibers $X_s$ of $X\to S$. We say that the set $E$ is {\em bounded} if 
a semi-analytic Stein compactum $K$ over $S$ exists together with a coherent sheaf  $\cG$ in a neighbourhood of $X\times_S K$
such that $E$ is contained in the set of isomorphism classes defined by the analytic restrictions of the sheaf $\cG$ to 
the fibers of $X\times_S K\to K$. A set of complex subspaces of the fibers of $X\to S$ is called  {\em bounded} if the set  of isomorphism classes of their structure sheaves is bounded.
\end{Def}
(We will be loose on the above terminology and often say ``sheaves'' instead of ``isomorphism classes of sheaves''.) 

In the above definition it is clear that $E$ can be viewed as set of sheaves defined on (some of) the fibers of 
$X\times_S K\to K$.

Let $X$ be an analytic space over $S$ and $\cF$ a coherent sheaf on $X$. If $T \to S$ is a morphism, we will write as usual $X_T:= X\times_S T$ and $\cF_T$ for the base change. The projections $X_T\to T$ will be denoted by $p_T$.

When speaking of morphisms defined on Stein compacta $K$ or sheaves over $K$ we mean of course that such objects are defined on some analytic space containing $K$.

\begin{Rem}\label{noether} 
Suppose $X$ is an analytic space 
over a reduced analytic space $S$ and $\cF$ a coherent sheaf on $X$
whose support is proper over $S$.
Then there exists a nowhere dense closed analytic subspace $T$ of $S$
such that over $S\setminus T$ the sheaf
 $\cF_{S\setminus T}$ is flat and base change holds for the sheaves $R^q{p_{S\setminus T}}_*(\cF)$ for all $q$. In this case we will simply say that base change holds for $\cF$ over $S\setminus T$.
\end{Rem}

Indeed, by Frisch's theorem on generic flatness, there exists a nowhere dense analytic subset $T'\subset S$ such that 
the sheaf  $\cF$ is flat on $S\setminus T'$. 
We apply Grauert's semicontinuity theorem and get a  nowhere dense  analytic subset $T''$ of the  Zariski open subset  $S\setminus T'$ of $S$ such that the sheaves $R^q{p_{S\setminus T'\setminus T''}}_*(\cF)$ are locally free and base change holds over $S\setminus T'\setminus T''$. Now by Hironaka's flattening theorem,  \cite{Hir75}, after some proper modification $S'\to S$ not affecting $S\setminus T'$ the sheaf $\cF_{S'}$ modulo $S'$-torsion becomes flat over $S'$. This implies that 
$T:=T'\cup T''$ is a closed analytic subspace of $S$. 

In fact a proof of this remark may be given using an older result of Kiehl-Verdier and Schneider, see \cite{BS} 3.4.1, and avoiding Hironaka's flattening theorem.

\begin{Rem}\label{echivalenta} 
Our definition is equivalent to Grothendieck's via GAGA when $X$ is complex projective.
\end{Rem}

Indeed, if a set of algebraic coherent sheaves on $X$ is bounded, then by Grothendieck's criterion, these sheaves appear as quotients of the one and the same sheaf $\cO(-n)^N$ and their Hilbert polynomials range in a finite set. Thus they will give points in a finite set of projective $\Quot$ schemes, hence the boundedness of the associated set of coherent analytic sheaves. 

Conversely, if a set of  coherent analytic sheaves on $X$ is bounded in our sense, then these sheaves will be fibers of some flat family over a semi-analytic Stein compactum $K$. Then by the same argument as before they will appear as quotients of a uniform sheaf of the type  $\cO(-n)^N$. As their Hilbert polynomials range in a finite set again, the existence and projectivity of the $\Quot$ scheme may be used to deduce the boundedness of the set of associated algebraic coherent sheaves.

The following proposition gathers some basic properties of bounded sets of coherent sheaves, which will be needed in the sequel.

\begin{Prop}\label{limitata} 
Let $X$ be an analytic space 
 over an analytic space $S$ and $E$, $E'$ two bounded sets of isomorphism classes of sheaves on the fibers of $X\to S$. 
 Suppose that there exists a closed subset $L$ of $X$ proper over $S$ such that the supports of the sheaves from $E$ or from $E'$ are contained in $L$.
 Then the following sets are bounded as well:
\begin{enumerate}
 \item The sets of kernels, cokernels and images of sheaf homomorphisms $\cF\to \cF'$, when the isomorphism
 classes of $\cF$ and $\cF'$ belong to $E$ and $E'$ respectively.
\item The set of isomorphism classes of extensions of $\cF$ by $\cF'$, for
  $\cF$ and $\cF'$ as above.
\item The set of isomorphism classes of tensor products $\cF\otimes\cF'$.
\end{enumerate}
\end{Prop}

\begin{proof}
It is clear that one can find a  base extension $K\to S$ from a semi-analytic Stein compactum together with coherent sheaves $\cG$ and $\cG'$ on $X_{K}:=X\times_S K$ such that $E$ and $E'$ are contained in the set of isomorphism classes defined by the sheaves $\cG$  and $\cG'$ on the fibers of $X_K\to K$. We may assume $K$ smooth and connected. Moreover the support of $\cG$ or the support of $\cG'$ is proper over $K$.
\begin{enumerate}
 \item 
 By Remark \ref{noether} there exists an analytic subset $K_1\subset K$ such that 
 the sheaves $\cO_{X_{K}}$, $\cG$, $\cG'$, and  $\cH om(\cG,\cG')$
are flat over $K\setminus K_1$ and base change holds for these sheaves over $K\setminus K_1$.  Working by descending induction on 
$\dim K$  we only need to check boundedness for the family of possible kernels and cokernels over $K\setminus K_1$. 

 Let $T:=\P (((p_{K})_*\cH om(\cG,\cG'))^{\vee})$ and 
$((p_{K})_*\cH om(\cG,\cG'))^{\vee})_T\to \cL$ be the universal quotient. Here and throughout the paper we use the notation $\P(\cE)$ for the projective bundle associated to the sheaf $\cE$ in the sense of Grothendieck. 

The semi-analytic compactum $T$ is not Stein in general but we may cover it by a finite number of semi-analytic Stein compacta.

 We get a universal section 
$$\cO_T\to (((p_{K})_*\cH om(\cG,\cG'))^{\vee})_T^{\vee}\otimes \cL$$
which restricted over the part of $T$ lying over $K\setminus K_1$ induces a universal family of morphisms $\cG_{X_s}\to \cG'_{X_s}$. We look for a coherent sheaf $\cK$ on $X_T$ which extends to $T$ the family of kernels of these morphisms. We may therefore reduce ourselves to the situation of a proper morphism $T\to K$ together with a section $\sigma$ of 
$((p_T)_*\cH om(\cG_T,\cG'_T))^{\vee\vee}$.
Let $\cT$ and $\cC$ be the kernel and the cokernel of the natural morphism  
$$(p_T)_*\cH om(\cG_T,\cG'_T)\to((p_T)_*\cH om(\cG_T,\cG'_T))^{\vee\vee}$$
of coherent sheaves on  $T$. Let further $U$ be a Stein open neighbourhood of some point $t\in T$ and $f$ any  function in 
$\cA nn(\cC)(U)$. Then $f\sigma$ has a lift  to $(p_T)_* \cH om(\cG_T,\cG'_T)(U)=\cH om (\cG_T,\cG'_T)(X_U)$ which we denote by $f\sigma$ for simplicity.  Let now $f_i$ and $g_j$ run through two finite sets of generators of $\cA nn(\cC)_t$ and  $\cA nn(\cT)_t$ respectively and suppose that they still generate these sheaves at any point of $U$. On $X_U$ we define
$$\cK:=\cap_{f_i, g_j}\Ker(g_jf_i\sigma:\cG_U\to\cG'_U).$$
It is easy to check that the definition does not depend on the chosen generators, hence the existence of the desired sheaf
$\cK$ on the compact space $X_T$. 
For a global substitute of a cokernel one may consider
$$\cG'_U/\sum_{i,j} \Im (g_jf_i\sigma:\cG_U\to\cG'_U).$$
 Then the boundedness holds over the open set of $T_{K\setminus K_1}$  where the cokernel is flat 
and we continue by noetherian induction on $K$.

\item
When $\cO_{X_{K}}$, $\cG$, $\cG'$ are flat over $K$ and $\cE xt^1(p_{K};\cG,\cG')$ commutes with base change,
families of extensions of $\cG$ by $\cG'$ are parametrized by $H^0(K; \cE xt^1(p_{K};\cG,\cG'))$, cf. \cite{Lan83}, see also
\cite{BPS80} and \cite{Fle81}. Such a family of extensions comes from a global extension on $X_K$ when 
$H^2(K; (p_{K})_*\cH om(\cG,\cG'))=0$ which is the case in our situation. Considering $\P ((\cE xt^1(p_{K};\cG,\cG'))^{\vee})$ we obtain as before a semi-analytic compactum $T$ containing some nontrivial Zariski open subset which parametrizes a universal family of classes of extensions of $\cG$ by $\cG'$. Let $T'$ be the complement of this open subset, $t$ be some point of $T$ and $U$ some Stein neighbourhood of $t$. 
The universal section $\xi$ of $\cE xt^1(p_{T\setminus T'};\cG,\cG')^{\vee\vee}$  multiplied again by some function $gf$ on $U$ extends as a section $gf\xi$ of $\cE xt^1(p_{U};\cG,\cG')$ over $U\setminus T'$ and thus as an element of $\Ext^1(X_U;\cG,\cG')$, since $U$ is Stein and we have an exact sequence
$$\Ext^1(X_U;\cG,\cG')\to H^0(\cE xt^1(p_{U};\cG,\cG'))
\to H^2(U, {p_U}_*\cH om ( \cG,\cG'))$$
induced by the spectral sequence $H^p(\cE xt^q(p_{U};\cG,\cG'))\Rightarrow\Ext^{p+q}(X_U;\cG,\cG')$.

We want to use the elements $gf\xi$ in order to get some coherent sheaf on $X_U$ whose restriction to $X_{U\setminus T'}$
would be precisely the extension defined by $\xi$. For this we recall the construction of this extension. Let 
$0\to \cI_0\to\cI_1\to...$ an injective resolution of $\cG'$ and $\delta:\cI_0\to\cI_1$ the first map. Then one may view $\xi$ as a homomorphism $\cG\to \Im(\delta)$ and the induced extension is
$$0\to \cG'\to \Ker((\delta+\xi):\cI_0\oplus\cG\to\Im(\delta))\to \cG\to 0.$$

Over $U\setminus T'$ we have a commutative diagram with exact rows

$$
  \xymatrix{0\ar[r] & \Ker(\delta+gf\xi)\ar[d]\ar[r] 
& \cI_0\oplus\cG \ar[d]^{\id \oplus gf\id}\ar[r]^{\delta+gf\xi} & \Im(\delta)\ar[d]^{\id}\\
0\ar[r] & \Ker(\delta+\xi)\ar[r] 
& \cI_0\oplus\cG \ar[r]^{\delta+\xi} & \Im(\delta)    }
  $$ 
and we consider the image $\cE_{gf}$ of $\Ker(\delta+gf\xi)$ in $\cI_0\oplus\cG$
 by the vertical arrow $\id \oplus gf\id$. 
It is clear that $\cE_{gf}$ coincides with $\Ker(\delta+\xi)$ away from the vanishing locus of $gf$. Then
$\cE:=\sum_{i,j}\cE_{g_j f_i}$ provides the desired substitute. Here we have considered the functions $f_i$, $g_j$ as before. It is again easy to check that $\cE$ does not depend on the chosen systems of generators.
Although $\cE$ might not sit as middle term of a sheaf extension of $\cG$ by $\cG'$, it  exists globally on $X_U$ and gives sheaf extensions in each fiber of $p_{U\setminus T'}$.  

\item This is clear since $(\cG\otimes_{\cO_X}\cG')\otimes_{\cO_X}\cO_{X_s}\cong \cG_s\otimes_{\cO_{X_s}}\cG'_s$.
\end{enumerate}
\end{proof}

We mention one more boundedeness result in the K\"ahler case which is of independent interest. For this we will need the following lemma.

\begin{Lem}\label{reflexive}
Let $f:X'\to X$ be a desingularization of the normal space $X$ and $\cF$ a reflexive sheaf of rank $1$ on $X$. 
Let $\cF':=(f^*\cF)^{\vee\vee}$. Then an  invertible sheaf $\cL$ on $X'$ has the property $f_*\cL=\cF$ if and only if
there exists an effective divisor $D$ supported on the exceptional locus of $f$ such that $\cL\cong\cF'(D)$. 
\end{Lem}
\begin{pf}
 Suppose that the   invertible sheaf $\cL$ on $X'$ has the property $f_*\cL=\cF$. Then the map
$f^*\cF\cong f^*f_*\cL\to\cL$ vanishes on torsion and thus factorizes through $\cF'$. It is moreover clear that this map is an isomorphism away from the exceptional locus of $f$.

Conversely, if $D$ is an effective divisor supported on the exceptional locus of $f$, we get natural morphisms
$$\cF\hookrightarrow f_*f^*\cF\to f_*\cF'\hookrightarrow f_*\cF'(D)\hookrightarrow (f_*\cF'(D))^{\vee\vee}\cong\cF$$
the last isomorphism being a consequence of the unicity of the extension of a reflexive sheaf over a codimension $2$ locus; cf. \cite{BS} Prop. 8.3.5.
Since the global endomorphisms of $\cF$ are homotheties, we get our conclusion.
\end{pf}

\begin{Prop}\label{reflexive-limitate}
 Let $X$ be a compact K\"ahler normal space of dimension $d$. 
Then the set of reflexive sheaves of rank one of fixed homology Todd class $\tau_{d-1}$ on $X$ is bounded.
\end{Prop}
\begin{pf}
 Let $f:X'\to X$ be a desingularization  $X$ and suppose that the exceptional divisor $E$ of $f$ has $n$ components. Denote by
$\Pic(X'\setminus E,X')$ the group of invertible sheaves on $X'\setminus E$ which may be extended as invertible sheaves to $X'$. 
It is clear that $\Pic(X'\setminus E,X')$ is isomorphic to the group of reflexive sheaves of rank one on $X$. The restriction 
$r:\Pic(X')\to\Pic(X'\setminus E,X')$
induces an exact sequence of groups
$$\Z^n\to\Pic(X')\to\Pic(X'\setminus E,X')\to0.$$
Note however that the natural map $\Pic(X'\setminus E,X')\to \Pic(X')$, $\cF\mapsto\cF'$, defined in Lemma \ref{reflexive}
is not a group morphism in general.

Let $\cL$ be any invertible sheaf on $X'$. The class $\tau_{d-1}(f_*\cL)=f_*\tau_{d-1}(\cL)$ only depends on $r(\cL)$.
In order to see this one may use the following natural diagram in singular homology:
$$
  \xymatrix{ H_{2d-2}(E;\Q) \ar[r] &  H_{2d-2}(X';\Q) \ar[r] \ar[d]^{f_*} & H_{2d-2}(X',E;\Q) \ar[d]_{\cong} ^{f_*}\\
& H_{2d-2}(X;\Q) \ar[r]^{\cong} & H_{2d-2}(X, \Sing(X);\Q)}
$$
or the corresponding diagram in Borel-Moore homology, cf. \cite{Ful} 19.1.
Consider the set of reflexive sheaves of rank one of fixed homology Todd class $\tau_{d-1}=a$ on $X$, let $\cL$ be an invertible sheaf on $X'$ such that $\tau_{d-1}(f_*\cL)=a$ and consider the subset $S$ of elements of $\Pic(X')$ having the same homology Todd class  as $\cL$. 
Then $S$ is a finite union of components of $\Pic(X')$ and any reflexive sheaf of rank one of  homology Todd class $\tau_{d-1}=a$ on $X$ is of the form
$(f_*(\cL'))^{\vee\vee}$ for some $\cL'$ whose isomorphism class lies in $S$. Denote the Poincar\'e invertible sheaf on $X'\times S$ by $\cP$. For each $s\in S$ the natural morphism $((f_S)_*\cP)_s\to (f_{S,s})_*(\cP_s)$ is an isomorphism over the regular part of $X$. Moreover for $s\in S$ general, that is away from a bad subset $S_1$, we have 
$(\cH om_{\cO_{X'\times S}}(\cP,\cO_{X'\times S}))_s\cong \cH om_{\cO_{X'_s}}(\cP_s,\cO_{X'_s})$  and the desired boundedness follows in the usual way by inductively constructing families of sheaves, first over $X\times S$, then over $X\times S_1$ and so on.
\end{pf}


\section{A boundedness criterion}\label{main}

Our main result is the following boundedness criterion.

\begin{Th}\label{principala}
Let $(X,\omega)$ be a K\"ahler complex manifold and $F$  a set of isomorphism classes of coherent sheaves on  $X$. 
 Then the set $F$ is bounded if the following conditions are fulfilled:
\begin{enumerate}
\item There exists a bounded set $G$ of classes of coherent sheaves on  $X$
 such that each element of $F$ is a quotient of an element of $G$.
 \item The supports of the sheaves of $F$ are contained in some compact subset of $X$. 
\item The degrees of the sheaves of $F$ are bounded from above. 
\end{enumerate}
\end{Th}

In general we shall call a set $F$ of isomorphism classes of coherent sheaves on the fibers of $X\to S$ {\em dominated} if it satisfies the first condition of the criterion.

%

Recall that  $\cF$ is called {\em pure of dimension} $d$ if $\cF$ as well as all its non-trivial coherent subsheaves are of dimension $d$. 
We introduce the following notation. As in Section 2 we denote by $\cN_r(\cF)$  the sheaf of sections of a coherent sheaf $\cF$ whose support have dimension less than $r$ and $\cF_{(r)}:=\cF/\cN_r$. If $\cF$ is of dimension $d$ then
$\cF_{(d)}$ is pure of dimension $d$.

In the proof we will repeatedly use the following strengthening of Bishop's theorem:

\begin{Th}\label{Lieberman}(\cite{Lie77})
A subset $C$ of the cycle space $C(X)$ of a complex space $X$ is relatively compact in $C(X)$ if and only if the cycles parameterized by $C$ all lie in a compact subset of $X$ and their volumes are uniformly bounded. 
\end{Th}

The proof of Theorem \ref{principala} will consist in the reduction to the following Lemma.

\begin{Lem}\label{tehnica}
 Let $d$ be a non-negative integer, 
$(X,\omega)$ as above and $S$ a reduced analytic space. 
Let further 
$W\subset X\times S$ be a reduced  analytic subset of $ X\times S$ such that the second projection $W\to S$ is proper, equidimensional of relative dimension $d$ with irreducible general fibers  and
 $F$ a dominated set of classes of pure $d$-dimensional sheaves  $\cF$ on the 
irreducible
fibers  of $W$ over some relatively compact subset of $S$. 
Then:
\begin{enumerate}
\item The set of classes $\tau_{d}(\cF)$ is finite.
\item The degrees $\deg_{d-1}(\cF)$ are bounded from below.
\item If the degrees $\deg_{d-1}(\cF)$ are bounded from above, then the set $F$ is bounded.
\end{enumerate}
\end{Lem}

\begin{proof}
By Hironaka's desingularization and flattening theorems we may assume that $S$ is smooth and irreducible containing a semi-analytic Stein compactum $A$,  $W$ is flat over $S$ and that there exists a coherent sheaf $\cG$ on $W$ flat over $S$ such that each sheaf $\cF$ in $F$ is a quotient of some $\cG_s$, $s\in A$. For instance in order to obtain  $W$ flat over $S$ we consider the locus $S_1\subset S$ of non-flatness of $W\to S$, flatten $W\to S$ to obtain $W'\to S'$, deal with $W_{S_1}\to S_1$ separately by noetherian induction and consider the resulting space over the disjoint union of $S'$ and $S'_1$. In the same way we may replace $\cG$ by $\cG_{(\dim W)}$ since the sheaves $\cF$ are pure.
Moreover the arguments of \cite{HL} 1.1.6-14 adapt to our case to show that for general $s\in S$ the fibers $\cG_s$ are pure as well. Hence as before we may suppose that each $\cG_s$ we are concerned with is pure.

For the first assertion consider the exact sequence 
$$0\to\cE\to\cG_s\to\cF\to0$$
exhibiting $\cF$ as a quotient of $\cG_s$. Then $\tau_{i}(\cG_s)=\tau_{i}(\cF )+\tau_{i}(\cE)$ 
and for  $i=d$ both $\tau_{i}(\cF )$ and $\tau_{i}(\cE)$ are positive (or zero).
In fact these classes correspond to integer combinations of $r$-dimensional irreducible components of $W_s$. 
Now the classes $ \tau_{i}(\cG_s)$ are independent of $s\in S$
by Proposition \ref{plat}.
Since the $d$-degrees of the sheaves $\cF$ are bounded from above by the $d$-degrees of the $\cG_s$, Bishop's theorem implies that the set of classes $\tau_{d}(\cF)$ must be finite.

For the next assertions we will need some further reductions: 
take an embedded desingularization $\tilde W\to W$ of $W\subset X\times S$ and a flattening $\hat W\to \tilde W$ of $\cG_{\tilde W}$ over $\tilde W$. Thus $\hat{\cG}:=\cG_{\hat W}/\Tors \cG_{\hat W}$ is locally free on $\hat W$.
Denote by $p:\hat W\to W$ the projection.
Notice that the topological invariants of the general fiber of $\hat W\to S$ are independent of the sheaves $\cF$.
The general fibers of $\hat W\to S$ are smooth and connected. Take now a sheaf $\cF$ on some $W_s$ 
and suppose that the corresponding $\hat W_s$ is smooth and connected. 
The projection $p_s:\hat W_s\to W_s$ factors through the normalization $g:W_s'\to W_s$ as 
$$
  \xymatrix{\hat W_s \ar[r]^{f} & W_s'\ar[r]^{g} & W_s. }
$$ 
Let $\cE$ be the kernel of $\cG_s\to \cF$, $\hat{\cE}$ the image of the composition 
$p_s^*\cE\to p_s^*\cG_s\to\hat{\cG_s}$ and $\cC$ the cokernel of $\cE\to (p_s)_*\hat{\cE}$. 
We have exact sequences of $\cO_{W_s}$-modules:
$$
0\to\cE\to\cG_s\to\cF\to0,
$$
$$
0\to\cE\to (p_s)_*\hat{\cE}\to\cC\to0
$$
hence $(p_s)_*(\tau_{d-1}(\hat{\cE}))=\tau_{d-1}((p_s)_*\hat{\cE})=\tau_{d-1}(\cE)+\tau_{d-1}(\cC)$.
Notice that $g^*\cF$ is pure but $p_s^*\cF$ might have torsion and its support is contained in the exceptional divisor $E$ of $f$. Let then $\hat\cE^{sat}$ be the saturation of $\hat\cE$ in $\hat\cG_s$, i.e. the kernel of the composition of morphisms
$\hat\cG_s\to\hat\cG_s/\hat\cE\to (\hat\cG_s/\hat\cE)/\Tors(\hat\cG_s/\hat\cE)$.
Thus the sequence
$$
0\to\hat\cE^{sat}\to\hat\cG_s\to p_s^*\cF/\Tors(p_s^*\cF)\to0
$$
is exact. 
Set $\hat\cF:=p_s^*\cF/\Tors(p_s^*\cF)$ and $\cL:=(\bigwedge^{\rank\cE}\hat\cE^{sat})^{\vee\vee}$. It is an invertible subsheaf of $\bigwedge^{\rank\cE}\hat\cG_s$ whose first Chern class $[c_1(\cL)]$ equals that of $\hat\cE^{sat}$. 

Let now $\hat\P:=\P(\bigwedge^{\rank\cE}\hat\cG_s)$ be the projectivized bundle of $\bigwedge^{\rank\cE}\hat\cG_s$ and $\P$ the irreducible component of $\P(\bigwedge^{\rank\cE}\cG_s)$ which covers $W_s$. We have a commutative diagram of natural morphisms
$$
 \xymatrix{\hat\P \ar[r]^{\hat p_s} \ar[d]^{\hat q} & \P \ar[d]^{q} \\
\hat W_s \ar[r]^{p_s} & W_s. }
$$ 
The inclusion $\cL\hookrightarrow\bigwedge^{\rank\cE}\hat\cG_s$ gives a section $\phi\in H^0(X; \bigwedge^{\rank\cE}\hat\cG_s\otimes\cL^{\vee})\cong H^0(\hat\P; \cO_{\hat \P}(1)\otimes \hat q^*\cL^{\vee})$ which vanishes on some irreducible divisor $\hat D$ on $\hat \P$. 
In fact $\hat D$ corresponds to the canonical embedding of $\P(\bigwedge^{\rank\cE}\hat \cG_s/\cL)$ in $\hat \P$. Denote by $D$ the image of $\hat D$ in $\P$.

Set $n:=\rank (\bigwedge^{\rank\cE} \cG_s)$ and let  $\omega$ be the restriction of the
K\"ahler form to $W_s$. One may choose Chern forms $\eta_{\P}$, $\eta_{\hat\P}$  of the tautological line bundles $\cO_{\P}(1)$ and $\cO_{\hat\P}(1)$ respectively which are positive on the fibers of $q$ and $\hat q$ respectively, \cite{Bi83} Lemma 4.19. Notice that $(p_s^*\omega)^{d-1}$ vanishes on the exceptional divisor $E$. Replacing $\eta_{\P}$, $\eta_{\hat\P}$ by some small multiples of theirs if necessary we obtain a semi-positive $(d+n-2, d+n-2)$-form  $\Omega:=q^*\omega^{d-1}\wedge  \eta_{\P}^{n-1}+q^*\omega^{d}\wedge \eta_{\P}^{n-2}$ on $\P$.  
 It is easily seen that $\Omega $ is strictly positive over the regular locus of  $(\bigwedge^{\rank\cE} \cG_s)^{\vee\vee}$ on $W_s$. 
 Put $C':=\int_{\hat\P}\hat q^*\phi\wedge \eta_{\hat\P}^{n-1}$, where $\phi$ is a $2d$-differential form representing the fundamental class on $\hat W_s$.
Then for the  volume of $D$ with respect to  $\Omega$ we obtain:

$$
\vol (D)= [\Omega]\frown (\hat p_s)_*\hat D=
(\hat p_s)_*(\hat p_s^*([\Omega])\frown \hat D)
=\hat p_s^*([\Omega])([\eta_{\hat\P}]-\hat q^*[c_1( \cL)])=
$$
$$
= (\hat p_s^*[\Omega])[\eta_{\hat\P}]-C'(p_s^*([\omega^{d-1}])[c_1( \cL)]=
(\hat p_s^*[\Omega])[\eta_{\hat\P}]-C'(p_s^*([\omega^{d-1}])[c_1(\hat \cE^{sat})])=
$$
$$
=(\hat p_s^*[\Omega])[\eta_{\hat\P}]+\frac{C'}{2}(p_s^*([\omega^{d-1}])c_1(\hat W_s)-C'(p_s^*([\omega^{d-1}])\frown \tau_{d-1}(\hat\cE^{sat})=
$$
$$
=C-C'[\omega^{d-1}]\frown \tau_{d-1}((p_s)_*(\hat \cE))=   
C-C'\deg_{d-1}((p_s)_*(\hat \cE))
\leq C-C'\deg_{d-1}(\cE),
$$
where $C$ and $C'$ are constants not depending on $\cF$ with $C'>0$, cf. \cite{Tel08} 2.1 for a similar computation in the smooth case. 
This proves the second assertion.

The same formula shows that an upper bound on $\deg_{d-1}(\cF)$ leads to an upper bound on $\vol (D)$. If we knew that $\Omega$ was strictly positive, then  Bishop's theorem  would imply that the analytic cycles $D$ may be seen as points in a part $T$ proper over $S$ of the relative cycle space of $\P((\bigwedge^{\rank\cE} \cG))$ over $S$.  This would allow to deduce a sequence of  boundedness results for the divisors $\hat{D}$, then for the sheaves $\cL$, $\hat{\cF}$ and eventually for $\cF$.

As it stands, we will still be able to bound the class of $\hat D$ modulo divisors supported over the exceptional  locus of $f:\hat{W_s}\to W'_s$. 
We will show that the lower bound on $\deg_{d-1}(\cE)$ allows for only a finite number of values of  $\tau_{d-1}(f_*\cL)$. Thus by Proposition \ref{reflexive-limitate} the sheaves of the type $(f_*\cL)^{\vee\vee}$
will belong to a bounded family, which in turn may be used to bound the set of divisors $\hat{D}$.
We consider for simplicity the case of a fiber $W_s$ but it will be clear that the arguments may be adapted to families. 

Note first that by the projection formula it follows that the class of $g^*\omega$ on $W_s'$ satisfies the hypothesis of the  Demailly-P\u aun  criterion, \cite{DePa04} Main Theorem 0.1.
In particular  $\omega':=g^*\omega+dd^c\phi$ is a K\"ahler form on $W_s'$ for some suitable smooth function $\phi$ on $W_s'$.


We have denoted by $E$ the exceptional divisor on $\hat W_s$ which is contracted by the map $f:\hat{W_s}\to W'_s$. Let $E_{\hat\P}:=\hat{q}^{-1}(E)$ be its pull-back to $\hat{\P}$.
We may compute the volume of the hypersurface $\hat D$ with respect to the form $\hat\Omega:= \hat q^*f^*(\omega')^{d-1}\wedge  \eta_{\hat\P}^{n-1}+\hat q^*f^*(\omega')^{d}\wedge \eta_{\hat\P}^{n-2}$ which is strictly positive on $\hat\P\setminus E_{\hat\P}$.  As before we get
$$\vol(\hat D)\leq C-C'\deg_{d-1}(\cE).$$
We  deduce from this inequality  that the set of cohomology classes in $H^2(\hat\P\setminus E_{\hat\P})$ of integration currents along hypersurfaces of type $\hat D\setminus E_{\hat\P}$ is uniformly bounded. 

These classes are the Chern classes of the restrictions of the line bundles  $\cO_{\hat\P}(1)\otimes \hat q^*\cL^{\vee}$ to $\hat\P\setminus E_{\hat\P}$. They lie in a compact set in $H^2(\hat \P\setminus E_{\hat\P}; \R)$ and consequently the following commutative diagrams

$$ 
\xymatrix{ & H^2(\hat \P;\Q) \ar[d]_{\cong} \ar[r] &  H^2(\hat \P\setminus E_{\hat\P}; \Q)\ar[d]_{\cong}\\
H_{2(d+n-2)}(E_{\hat\P};\Q) \ar[r] & H_{2(d+n-2)}(\hat \P;\Q) \ar[r]  & H^{BM}_{2(d+n-2)}(\hat \P\setminus E_{\hat\P};\Q) }
$$
and
$$
\xymatrix{ 
 & H^2(\hat \P;\Q)  \ar[r] &  H^2(\hat \P\setminus E_{\hat\P}; \Q) \\
& H^2(\hat W_s;\Q) \ar@{^{(}->}[u] \ar[d]_{\cong} \ar[r] &  H^2(\hat W_s\setminus E; \Q) \ar@{^{(}->}[u] \ar[d]_{\cong} \\
H_{2(d-1)}(E;\Q) \ar[r] & 
  H_{2d-2}(\hat W_s;\Q) \ar[r] \ar[d]^{f_*} & H^{BM}_{2d-2}(\hat W_s\setminus E;\Q) \ar[d]_{\cong} ^{f_*} \\
& H_{2d-2}(W'_s;\Q) \ar[r]^{\cong} & H^{BM}_{2d-2}(W'_s\setminus f(E);\Q)
}
$$
show that 
 $\tau_{d-1}(f_*(\cL))=\tau_{d-1}((f_*(\cL))^{\vee\vee}) $ attains only a finite number of values in $H_{2d-2}(W'_s;\Q)$. By Proposition \ref{reflexive-limitate} the sheaves of the type $(f_*\cL)^{\vee\vee}$
will thus belong to a bounded family. Hence  the set of  sheaves  
$(f^*(f_*\cL)^{\vee\vee})^{\vee\vee}$ on $\hat W_s$ will also be bounded. 
All morphisms of the following commutative diagram of coherent $\cO_{\hat W_s}$-modules are injective

$$
\xymatrix{ \cL  \ar[d] &  (f^*f_*\cL)^{\vee\vee} \ar[d]\ar[r]\ar[l] & (f^*(f_*\cL)^{\vee\vee})^{\vee\vee}  \ar[d] \\
\bigwedge^{\rank\cE} \hat\cG_s & (f^*f_*\bigwedge^{\rank\cE} \hat\cG_s)^{\vee\vee} \ar[r]\ar[l] &
 (f^*(f_*   \bigwedge^{\rank\cE} \hat\cG_s    )^{\vee\vee})^{\vee\vee}  .
}
$$
Now the sheaves $(f^*(f_*\cL)^{\vee\vee})^{\vee\vee} \cap \bigwedge^{\rank\cE} \hat\cG_s $ lie in a bounded set and coincide with the sheaves $\cL$ on $\hat W_s\setminus E$. Since the latter are 
 saturated in $\bigwedge^{\rank\cE} \hat\cG_s$ we get inclusions 
 $((f^*(f_*\cL)^{\vee\vee})^{\vee\vee} \cap \bigwedge^{\rank\cE} \hat\cG_s)^{\vee\vee} \subset \cL\subset \bigwedge^{\rank\cE} \hat\cG_s$. 
 Now the degree of $\cL$ with respect to some fixed metric on $\hat W_s$ is upper bounded by the slope of the maximal destabilizing subsheaf of $ \bigwedge^{\rank\cE} \hat\cG_s$, which only depends on $\cG_s$. We thus obtain  the boundedness of the set of invertible sheaves $\cL$. This is equivalent to the boundedness of the sheaves of the form $\det \hat\cF$.


Let $r$ be the rank of $\cF$. One checks directly then that
 $\hat\cF$ is the image of the composition of the natural maps
$$
\hat\cG_s \to \cH om(\bigwedge^{r-1} \hat\cG_s, \bigwedge^{r}\hat\cG_s) \to \cH om(\bigwedge^{r-1} \hat\cG_s, \det(\hat\cF)).
$$
Since all terms above sit in bounded sets we get the boundedness of the sheaves $\hat\cF$ by Proposition \ref{limitata}. Let $\tilde \cF$ be a family containing these sheaves. We suppose that the family $\tilde \cF$ is over $S$ for simplicity of notation.
The composition
$$
\cG_s\to (p_*p^*\cG)_s\to (p_s)_*((p^*\cG)_s) \cong (p_s)_*(p_s)^*\cG_s \to (p_s)_* (\tilde \cF_s) \cong ( p_s)_*(p_s^*\cF/\Tors(p_s^*\cF))
$$
factorizes as
$$
\cG_s\to (p_*p^*\cG)_s\to (p_*\tilde \cF)_s  \to (p_s)_* (\tilde \cF_s) $$ 
and factorizes also through 
the surjection $\cG_s\to \cF$. Since $\cF$  and  also $(p_*\tilde \cF)_s$ for general $s$ are pure, the image of this composition is isomorphic to both $\cF$ and $(p_*\tilde \cF)_s$.

The boundedness of the set $F$ now follows.
\end{proof}
We will also need the following consequence of Theorem \ref{Lieberman}.

\begin{Lem}\label{chow}
 Let  $X$ be a K\"ahler manifold, $r$ an integer and 
$F$ a set of compact reduced  subspaces of  $X$ of  bounded degree  and all of whose components are of dimension $r$ and contained in a fixed compact subset of $X$. Then $F$ is bounded.
\end{Lem}
\begin{pf}
The part $S$ of the Barlet space parameterizing the analytic cycles of dimension $r$ on $X$  of degree bounded by some constant  $M$ and contained in some fixed compact subset of $X$ is compact 
by \ref{Lieberman}. Moreover by \cite{Bar75} Thm. 1, p. 38 the set $W:=\{(x,s)\in X\times S \ | \ x\in |Z_s|\}$ is analytic and closed in $X\times S $, see also \cite{Fuj78}. Here $|Z_s|$ denotes the support of the cycle corresponding to $s\in S$. Furthermore, there exist positive integers
$n_j$ associated to the irreducible components $W_j$ of $W$ and a dense Zariski open subset $V\subset S $ such that for each $s\in V$ the multiplicities of the irreducible components of $Z_s$ contained in $W_j$ are precisely $n_j$. 

Consider the (reduced) subspace $W'\subset W$ consisting of those components $W_j$ for which $n_j=1$. 
The general fibers of $W'\to  S $ are reduced and cover part of our set $F$. We deal with the rest by noetherian induction on $S$.
\end{pf}


We can now prove  Theorem \ref{principala}.
\begin{pf}
Let $F$ be a set of isomorphism classes of sheaves $\cF$ dominated by a bounded set $G$ on $X$ and such that their supports are contained in a fixed compact subset of $X$. We suppose that for all $n\in\Z$ the
 degrees $\deg_{n}(\cF)$ are bounded from above.

We will show that the set $F$ is bounded 
working by induction on $r:=\max_{[\cF]\in F}\dim \Supp\cF$. 

For $r<0$ there is nothing to prove. 

Suppose that $r\geq 0$ and that the  statement holds for all sets $F'$ as above and with $\max_{[\cF]\in F'}\dim \Supp\cF<r$.

By Lemma \ref{chow} the set of the $r$-dimensional parts of the supports $\Supp\cF$ hence also the set of their corresponding ideal sheaves $\cI$ are bounded. The sets of ideals of type $\cI^j$ are equally bounded by Proposition \ref{limitata}. Using the boundedness assumption
on the degrees of $\cF$  and the devissage $\cI^j\cF \subset \cI^{j-1}\cF\subset ...\subset \cI\cF\subset\cF$ we find  the existence of  some integer $k$ such that $\dim \Supp(\cI^k\cF)<r$ for all $\cF$ in $F$. 
Let $\cF_j:=(\cI^{j-1}\cF)/(\cI^j\cF)$. The sets of sheaves $\cF_j$ as well as $\cI^k\cF$ are dominated and
$$\deg_d \cF=\sum_{j=1}^k\deg_d\cF_j +\deg_d (\cI^k\cF)$$
for all $d\in \Z$.

Let 
$Y_1, ..., Y_n$ denote the irreducible  components of $\Supp(\cO_X/\cI)$ and for  $1\le j\le k$ and $1\le i\le n$ set 
$$\cF_{j,i}:=\sum_{1\le l\le n, \ l\neq i}\cI_{Y_l}\cF_j$$ 
and $\cF'_j:=\Im(\oplus_i\cF_{j,i}\to\cF_j)$. Then the sets of sheaves $\cF'_j$ and $\cF_j/\cF'_j$ are dominated, $\dim\Supp(\cF_j/\cF'_j)<r$ and
 $$
\deg_d\cF=\sum_{j=1}^k\deg_d\cF'_j +\sum_{j=1}^k\deg_d(\cF_j/\cF'_j) +\deg_d (\cI^k\cF)
$$
 for all $d\in \Z$.

By applying the functor $-_{(r)}$ to the sequence of maps  $\oplus_i\cF_{j,i}\to\cF'_j\to \cF_j$ one sees that
$(\cF'_j)_{(r)}\cong\oplus_i(\cF_{j,i})_{(r)}$, hence the sheaves $(\cF'_j)_{(r)}$ are direct sums of pure sheaves of dimension $r$ each supported on some component $Y_i$ of  $\Supp(\cO_X/\cI)$. Moreover the formula 

\begin{equation}\label{suma}
\deg_d\cF=\sum_{j,i}\deg_d(\cF_{j,i})_{(r)} +\sum_{j=1}^k\deg_d\cN_r(\cF'_j)+\sum_{j=1}^k\deg_d(\cF_j/\cF'_j) +\deg_d (\cI^k\cF)
\end{equation}
holds for all $d\in \Z$.
 
We will next use this formula for $d=r-1$ and apply Lemma \ref{tehnica} to show that the sheaves $(\cF_{j,i})_{(r)}$ form a bounded set. Lemma \ref{chow} gives us an analytic family of supports $W\subset X\times S$ parametrized by some compact analytic space $S$ and we may view the sheaves $(\cF_{j,i})_{(r)}$ as sheaves on the fibers of $W\to S$. We may also suppose $S$ and $W$ reduced and irreducible. 
The first two assertions of Lemma \ref{tehnica} show that the $(r-1)$-degrees of the sheaves $(\cF_{j,i})_{(r)}$ are bounded from below. Using this and the upper bound on $\deg_{r-1}\cF$ in formula (\ref{suma}) imply that they are also bounded from above.
Then the third assertion of Lemma \ref{tehnica} says that after some base change $K\to S$ to a semi-analytic Stein compactum $K$ the sheaves  $(\cF_{j,i})_{(r)}$ occur as fibers of a family $\cH$ over $W_K$. Using  generic flatness we see that our sheaves are also fibers of some family over $X\times K'$ for some possibly new semi-analytic Stein compactum $K'$ and our claim is proven.

 Let  $G_j$ be a bounded  set of sheaves $\cG_j$ dominating the set of sheaves $\cF'_j$. Then
 the kernels $\cK_j$ of the surjections $\cG_j\to (\cF'_j)_{(r)}$ also form a bounded family, for $\cG_j\in G_j$. From the diagram
$$
  \xymatrix{0\ar[r] & \cK_j\ar[r]\ar[d] & \cG_j \ar[r] \ar[d]&  (\cF'_j)_{(r)}\ar[r] \ar[d]^{\id} & 0\\
0\ar[r] & \cN_r(\cF'_j)\ar[r] & \cF'_j\ar[r] &  (\cF'_j)_{(r)}\ar[r] & 0    }
  $$ 
we see that the set of sheaves $\cN_r(\cF'_j)$ is dominated as well.

The $2k+1$ sheaves of dimension $<r$ appearing in formula (\ref{suma}) may be each considered on a different copy of $X$ thus giving a sheaf on the disjoint union $X^{(2k+1)}$ of these copies. These sheaves on $X^{(2k+1)}$  have bounded degrees, are dominated and have their supports contained in a compact subset of $X^{(2k+1)}$, hence they form a bounded set by the induction hypothesis.
 But then each of the sets of sheaves $\cN_r(\cF'_j)$, $\cF_j/\cF'_j$, $\cI^k\cF$ is bounded on $X$.
Now the exact sequences 
$$0\to\cN_r(\cF'_j)\to\cF'_j\to(\cF'_j)_{(r)}\to 0,$$
$$0\to\cF'_j\to\cF_j\to \cF_j/\cF'_j\to 0,$$
$$ 0\to\cI^j\cF\to\cI^{j-1}\cF\to\cF_j\to 0$$
allow us to reconstruct $\cF$ and deduce the desired boundedness.
\end{pf}

\section{Corollaries}\label{sectioncor}

We start with some  direct consequences of Theorem \ref{principala}.

\begin{Cor}\label{cazul compact}
Let $(X,\omega)$ be a K\"ahler compact complex manifold and $F$  a set of isomorphism classes of coherent sheaves on  $X$. 
 Then the set $F$ is bounded if and only if the following two conditions are fulfilled:
\begin{enumerate}
\item The set $F$ is dominated.
\item The degrees of the sheaves of $F$ are bounded from above. 
\end{enumerate}
\end{Cor}

\begin{pf}
 By our discussion on Chern classes and since a semi-analytic Stein compactum has only finitely many connected components,
it is clear that the two conditions are necessary. Their sufficiency follows from our Theorem \ref{principala}.
\end{pf}

\begin{Cor}\label{cp}
 Let $(X,\omega)$ be a K\"ahler compact complex manifold and $F$ a dominated set of isomorphism classes of coherent sheaves on  $X$. 
 Then the following assertions are equivalent:
\begin{enumerate}
\item The set $F$ is bounded.
\item The set of Chern classes of the sheaves of $F$ is finite.
\item The set of homology Todd classes of the sheaves of $F$ is finite.
\item The degrees of the sheaves of $F$ are bounded. 
\item The degrees of the sheaves of $F$ are bounded from above. 
\end{enumerate}
\end{Cor}

We come to our main application: the compactness of the connected components of the Douady space of a compact K\"ahler manifold.

\begin{Cor}\label{compacitatea}
Let $X$ be a compact K\"ahler manifold, $\cG$ a coherent sheaf on $X$ and $b$ some real number. Then the Douady space $D(\cG)_{\leq b}$ of quotients of $\cG$ with degrees bounded from above by $b$ is compact. In particular the Douady space $D(\cG)_{\alpha}$ of quotients of $\cG$ with fixed homology Todd class equal to $\alpha\in H_*(X,\Q)$ is compact, especially the connected components of the whole Douady space $D(\cG)$ are compact.
\end{Cor}
\begin{pf}
By our boundedness criterion, the sheaves which are quotients of $\cG$ and whose degrees are bounded from above by $b$ are fibers over a semi-analytic Stein compactum $K$ of some family $\cF$ over $X\times S$ with $S$ smooth. Here we may suppose that $K$ is contained in $S$. By noetherian induction we may suppose that $\cF$ is flat over $S$. 

To any complex space $T$ over $S$ we associate the set
$\Hom_{X\times T}(\cG_T,\cF_T)$. 
This defines a contravariant functor which is represented by a linear space $V$ over $S$, cf. \cite{Fle81} 3.2. 
It is clear that we may find a finite number of semi-analytic Stein compacta in $V$ covering $K$ and such that  up to some multiplicative constant each non-zero morphism in  $\Hom_X(\cG,\cF_s)$, $s\in K$, is represented by some element in this union; see for instance the construction of the projective variety over $S$ associated to $V$ in \cite{Fi} 1.9. 
Let $K'$ be the union of these semi-analytic Stein compacta and $\cF'$ the image of the universal morphism restricted to $K'$. By flattening and noetherian induction again we may assume that $\cF'$ is flat over a neighbourhood $U$ of $K'$ and that each quotient of $\cG$ whose class is in $D(\cG)_{\leq b}$ is represented by some morphism $\cG\to\cF'_s$ for some $s\in K'$. The universal property of the whole Douady space  $D(\cG)$   of quotients of $\cG$ now gives us a  morphism $U\to D(\cG)$ whose restriction to $K'$ covers $D(\cG)_{\leq b}$. Thus 
$D(\cG)_{\leq b}$ is compact.
\end{pf}

Let now $X$ be a complex manifold, not necessarily compact, $d$ be a non-negative integer and $D_d(X)$ the Douady space of purely $d$-dimensional compact complex subspaces of $X$. Recall that $D_d(X)$ is a closed subspace of the whole Douady space $D(X):=D(\cO_X)$. The following corollary about the properness of the morphism ``Douady $\to$ Barlet'' (\cite{Bar75}, \cite{Mag07}) which was suggested to us by Daniel Barlet is now a direct consequence of our  boundedness criterion combined with Proposition \ref{plat} and Theorem \ref{Lieberman}.

\begin{Cor}\label{Dou-Ba}
If $X$ is a K\"ahler manifold and $D'_d$ the reduction of a connected component of the Douady space $D_d(X)$ of purely $d$-dimensional compact complex subspaces of $X$, then the natural morphism from $D'_d$ to Barlet's space $C_d(X)$ of $d$-dimensional cycles of $X$ is proper.
\end{Cor}

We finally present an application to the moduli space of torsion free rank one sheaves on a compact K\"ahler manifold.
Let $X$ be a connected compact complex manifold. It is immediately seen that any torsion free rank one sheaf on $X$ is simple. Moduli spaces for simple coherent sheaves on compact complex spaces were constructed in  \cite{KoOk89}. The torsion free rank one sheaves are thus parameterized by an open subset $\cM_1(X)$ of this space, cf. \cite{BaLeP87}. Since the parameterized objects have rank one, 
this subset is also separated, \cite[Proposition 6.6]{KoOk89}.  By fixing the total Chern class $c(\cF)=c$ of the sheaves to be parameterized we obtain a union $\cM_{1, c}(X)$ of connected components of $\cM_1(X)$.

\begin{Cor}\label{moduli}
If $X$ is a connected, compact, K\"ahler manifold and $c\in H^*(X,\Z)$ is a fixed total Chern class, then $\cM_{1, c}(X)$ is compact.
\end{Cor}
\begin{pf}
 If $\cF$ is a torsion free rank one sheaf on $X$ with $c(\cF)=c$ we have exact sequences
$$0\to\cF\to\cF^{\vee\vee}\to\cF^{\vee\vee}/\cF\to 0,$$
$$0\to \cF\otimes\cF^{\vee}\to\cO\to(\cF^{\vee\vee}/\cF)\otimes\cF^{\vee}\to0$$
and the Chern classes of $\cF^{\vee\vee}$ and of $(\cF^{\vee\vee}/\cF)\otimes\cF^{\vee}$ are completely determined by $c$.

Let $\Pic^{c_1}(X)$ and $\cP\to X\times \Pic^{c_1}(X)$ be the corresponding component of the Picard group of $X$ and the restriction of the Poincar\'e bundle. Let further $D_{c'}(X)$ be the Douady space of analytic subspaces of $X$ with total Chern class equal to $c':=c((\cF^{\vee\vee}/\cF)\otimes\cF^{\vee})$ and $\cD\subset X\times D_{c'}(X)$ the associated universal subspace.
To $\cD$ corresponds a family of ideal sheaves $\cI_{\cD}\subset \cO_{X\times D_{c'}(X)}$ which is flat over $D_{c'}(X)$. Then 
$\cI_{\cD}\boxtimes\cP$ is a family of torsion free sheaves of rank one on $X$ flat over $D_{c'}(X)\times \Pic^{c_1}(X)$. By the universal property of $\cM_{1, c}(X)$ one gets a morphism $D_{c'}(X)\times \Pic^{c_1}(X)\to\cM_{1, c}(X)$ which is also bijective. The conclusion follows now from our Corollary \ref{compacitatea}
 and the compactness of $\Pic^{c_1}(X)$.
\end{pf}


 \hrule \medskip
\par\noindent

Address:\\
Institut \'Elie Cartan, UMR 7502, 
Universit\'e de Lorraine, CNRS, INRIA,
Boulevard des Aiguillettes, B.P. 70239, 54506 Vandoeuvre-l\`es-Nancy Cedex, France \\ 
{\tt Matei.Toma@univ-lorraine.fr}

\end{document}